\documentclass[12pt, reqno]{amsart}
\usepackage{amsmath, amsthm, amscd, amsfonts, amssymb, graphicx, color}
\usepackage[bookmarksnumbered, colorlinks, plainpages]{hyperref}
\usepackage{float}
\usepackage{graphicx,epstopdf}
\usepackage{bbm}
\usepackage{mathrsfs}
\usepackage{dsfont}
\usepackage{amssymb}
\usepackage{csquotes}
\usepackage[dvips]{epsfig}
\usepackage{latexsym}
\usepackage{exscale}
\usepackage[latin1]{inputenc}

\usepackage{tikz}
\usepackage{tikz-cd}
\usetikzlibrary{matrix,arrows,decorations.pathmorphing}

\hypersetup{colorlinks=true,linkcolor=red, anchorcolor=green, citecolor=cyan, urlcolor=red, filecolor=magenta, pdftoolbar=true}

\newtheorem{theorem}{Theorem}[section]
\newtheorem{lemma}[theorem]{Lemma}
\newtheorem{proposition}[theorem]{Proposition}

\theoremstyle{definition}
\newtheorem{definition}[theorem]{Definition}

\theoremstyle{remark}

\numberwithin{equation}{section}

\newcommand{\norm}[1]{\lVert#1\rVert}

\newcommand{\abs}[1]{\left\lvert#1\right\rvert}
\newcommand{\N}{\mathbb{N}} 
\newcommand{\R}{\mathbb{R}} 
\newcommand{\C}{\mathbb{C}} 
\newcommand{\K}{\mathbb{K}} 

\newcommand{\Lin}{\mathcal{L}}
\newcommand{\BN}{\mathcal{BN}}
\newcommand{\U}{\mathcal{U}}
\begin{document}

\setcounter{page}{1}

\title{Ideal of hypercyclic operators that factor through $\ell^p$}

\begin{center}
\author[A. G. AKSOY, Y. PUIG  ]{Asuman G\"{u}ven AKSOY, Yunied Puig}
\end{center}

\address{$^{*}$Department of Mathematics, Claremont McKenna College, 850 Columbia Avenue, Claremont, CA  91711, USA.}
\email{\textcolor[rgb]{0.00,0.00,0.84}{aaksoy@cmc.edu}}

\address{$^{1}$Department of Mathematics, Claremont McKenna College, 850 Columbia Avenue, Claremont, CA  91711, USA.}
\email{\textcolor[rgb]{0.00,0.00,0.84}{puigdedios@gmail.com}}





\begin{abstract}
We study the injective and surjective hull of operator ideals generated by hypercyclic backward weighted shifts that factor through $\ell^p$. 
\end{abstract} \maketitle

\section{Introduction}
 Let $T$ be a linear and compact operator acting from a Banach space $X$ into a Banach space $Y$.  We set $\mathcal{L}(X,Y)$  to denote the normed vector space of all continuous operators from $X$ to $Y$ and $\mathcal{K}(X,Y)$ be the collection of all compact operators from $X$ to $Y$.   The theory of compact linear operators between two Banach spaces $X$ and $Y $ has  a classical core and is familiar to many. For example, the well-known  classical factorization theorem due to  W. B. Johnson \cite{Jon}  asserts that every compact operator $T$ from a Banach space $X$ to a Banach space $Y$ factor compactly through a Banach space  $Z$  provided $T=P\circ Q$ and $P\in \mathcal{K}(Z, Y)$ and $Q\in \mathcal{K}(X,Z)$. 
 As an application for operators of a classical characterization of compact sets due to  Grothendieck \cite{Gr}, we have that $T$ compact implies
\begin{equation}
\label{eq_echo_park}
T(U_X)\subseteq \left \{\sum_n\alpha_ny_n:(\alpha_n)\in U_{\ell^1}\right \}
\end{equation}
 where $U_X$ denotes the open unit ball in $X$, $y_n\in Y$ for all $n$ and $\norm{y_n}\to 0$ as $n\to \infty$.  There is a close connection between factorization of compact operators through a close subspace of $c_0$ and above  representation of compact sets.  In this direction we mention a result due to T. Terzio\u{g}lu:
 \begin{theorem}[\cite{Ter1}]
Let $T\in \mathcal{L} (X,Y)$ bounded linear operator between Banach spaces $X$ and $Y$. Then the following are equivalent:
\begin{enumerate}
\item $T \in \mathcal{K}(X,Y)$,
\item There exists a norm-null sequence $(x_n^*)$ in $X^*$ such that 
$$ ||Tx|| \leq \sup_{n}| \langle x, x_n^*\rangle | \quad \forall x \in X,$$
\item For some closed subspace $Z$ of $c_0$ there are compact operators $P \in \mathcal{K} (X, Z)$ and $Q\in \mathcal{K}( Z, Y)$ such that $T= Q \circ P$.
\end{enumerate}
\end{theorem}
 This factorization theorem has a number of important connections and consequences  which were discussed in detail in \cite{Asu}.   

 A linear and bounded operator $T$ acting on a separable infinite-dimensional Banach space $X$ is called \emph{hypercyclic} if there exists $x\in X$ such that $\{T^nx:n\geq 0\}$ is dense in $X$. It is well-known that hypercyclic operators can never be compact \cite{BaMa}. We concentrate our attention on operators that can be hypercyclic and satisfy an expression in the vein of (\ref{eq_echo_park}). More precisely, we are interested in studying operators satisfying
\begin{equation}
\label{eq_silver_lake}
T(U_X)\subseteq \left\{\sum_n\alpha_n y_n: (\alpha_n)\in U_{\ell^{q'}}\right\}
\end{equation}
where $(y_n)$ is a weakly $q$-summable sequence in $X$ with $1/q+1/q'=1$. These operators can be hypercyclic. Indeed, weighted backward shifts are examples of operators satisfying (\ref{eq_silver_lake}).  Since it is shown by H. Salas  in \cite{Sal} that every operator of the form $ \lambda I +$ (backward shift) is  hypercyclic and  due to the importance of weighted backward shifts in linear dynamics, we  now focus our interest on studying operators satisfying (\ref{eq_silver_lake}).

 In section 3 below, we introduce the ideal of $(w, p, q)$-backward nuclear operators, i.e., ~the ideal of operators generated by a weighted backward shift acting from $\ell^{q'}$ into $\ell^p$, denoted by $\BN_{(w, p, q)}$.  Using the standard techniques which were first systematized in \cite{Pie}, we show that the injective hull of $\BN_{(w, p, q)}$ is composed by operators satisfying conditions in the spirit of Terzio\u{g}lu (Theorem \ref{theorem_istanbul} and Theorem \ref{theorem_havana}). We  also prove the elements of the surjective hull of $\BN_{(w, p, q)}$ satisfy (\ref{eq_silver_lake}) (Theorem \ref{theorem_miami}). Moreover, we show the dual operator ideal of the surjective (injective) hull of $\BN_{(w, p, q)}$ is its injective (surjective) hull (Theorem \ref{theorem_orlando} and Theorem \ref{theorem_westpalmbeach}).

\section{Preliminaries}
In this section we define some  notations, concepts and mention some theorems that we will use later.

Let $\K=\R$ or $\C$ and $\mathcal{L}$ denotes the class of all linear and bounded operators between arbitrary Banach spaces. The identity map of $E$ is denoted by $I_E$.

An \emph{operator ideal} $\U$ is a subclass of $\Lin$ such that the components $\U(E, F):=\U \cap \Lin(E, F)$ satisfy the following conditions:

(1) $I_{\K}\in \U$,

(2) if $S_1, S_2\in \U(E, F)$ then $S_1+S_2\in \U(E, F)$,

(3) if $T\in \Lin(E_0, E), S\in \U(E, F)$ and $R\in \Lin(F, F_0)$, then $RST\in \U(E_0, F_0)$.\\

 Let $\U$ be an operator ideal, the map $\norm{\cdot}_\U: \U\to \R^+$ is called an \emph{$s$-norm on the operator ideal $\U$} with $0<s\le 1$ if the following are satisfied:

(1) $\norm{I_{\K}}=1$,

(2) $\norm{T_1+T_2}^s_\U\le \norm{T_1}^s_\U$ + $\norm{T_2}^s_{\U}$,

(3) if $S\in \Lin (X_0, X)$, $T\in \U (X, Y)$ and $R\in \Lin (Y, Y_0)$ then $\norm{RTS}_\U\le \norm{R} \norm{T}_\U \norm{S}$.\\

 If each $\U(X, Y)$ is a Banach space with respect to the ideal $s$-norm $\norm{\cdot}_\U$, then $(\U, \norm{\cdot}_\U)$ is called an \emph{$s$-Banach operator ideal}.

 Denote by $\U^{inj}, \U^{sur}, \U^{dual}$ and $\U^{reg}$ the injective, surjective, dual and regular hull of $\U$ respectively, in the sense of Pietsch, see chapter 4 \cite{PieID}.

Let $X^*$ be the dual space of $X$, $a\in X^*$ and $y\in Y$. Define $a\otimes y\in \Lin(X, Y)$ as $(a\otimes y) (x):=a(x)y$. Denote by $T^*$ the adjoint operator of any $T\in \Lin (X, Y)$ and $p'$ the conjugate of $p\geq 1$, i.e. $1/p+1/p'=1$.

Let $I$ be an arbitrary index set. Denote by $\ell_p(I)$, the \emph{Banach space of all absolutely p-summable scalar families} $x=(\epsilon_i)_{i\in I}$.

Recall that for $1\leq p \leq\infty$, the space of all \emph{weakly $p$-summable sequences} in a Banach space $X$ is defined as \[l^w_p(X)= \{(x_i)\in X^{\mathbb{N}}: (a(x_i))_i\in l_p, \forall a\in X^*\},\]

with norm 
\[ \norm{{(x_i)}}^w_p=\sup \{ \norm{{(a(x_i))_i}}_p: a\in X^*, \norm{{a}}\leq 1\}.
\]

 In the same way, the space of all \emph{weak* $p$-summable sequences} in $X^*$ is defined as \[l^{w*}_p(X^*)= \{(x_i)\in (X^*)^{\mathbb{N}}: (x_i(x))_i\in l_p, \forall x\in X\},\]

with norm 
\[ \norm{{(x_i)}}^{w^*}_p=\sup\{\norm{{(x_i(x))_i}}_p: x\in X, \norm{{x}}\leq 1\}.
\]
For more about these spaces see \cite{Gr1}.

\section{Results}

\subsection{The operator ideal of $(w, p, q)$-backward nuclear operators}
Classical $(r, p, q)$-nuclear operators, in the sense of Pietsch (18.1 \cite{PieID}) are always compact, hence not  hypercyclic.  As we are interested in hypercyclicity in connection with $(r, p, q)$-nuclear operators, we introduce the following.
\begin{definition}
Let $1\le p, q \le \infty$ such that $1/p+1/q\le 1$. An operator $T\in \Lin (X, Y)$ is called \emph{$(w, p, q)$-backward nuclear} if 
\[
T=\sum_{i=2}^\infty w_ia_i \otimes y_{i-1}
\]
\end{definition}
with $(w_i)\in \ell^\infty$, $(a_i)\in \ell^{w^*}_{q'}(X^*)$ and $(y_i)\in \ell^w_{p'}(Y)$.

Define 
\[
\norm{T}_{(w, p, q)}=\inf \norm{w}_\infty \norm{(a_i)}^{w^*}_{q'}\norm{(y_i)}^w_{p'},
\]
where the infimum is taken over all $(w, p, q)$-backward nuclear representations described above.

Denote by $\BN_{(w, p, q)}$ the class of all $(w, p, q)$-backward nuclear operators. Note that weighted backward shifts on $\ell^p$ are examples of $(w, p, p')$-backward nuclear operators, as they are defined as $B_w e_n:=w_n e_{n-1}, n\geq 1$ with $B_w e_0:=0$, where $(e_n)_{n\in \N}$ denotes the canonical basis of $\ell^p$ and  $B_w$ is the unilateral weighted backward shift operator.

\begin{theorem}
\label{theorem_guantanamo}
Let $1/s=1/p'+1/q'$, then $(\BN_{(w, p, q)}, \norm{\cdot}_{(w, p, q)})$ is an $s$-Banach operator ideal.
\end{theorem}

 In order to prove Theorem \ref{theorem_guantanamo} we need the following criterion provided in \cite{PieID}.

\begin{lemma}[Theorem 6.2.3 \cite{PieID}]
\label{lemma_criterion}
 Let $0<s\le 1$ and $\U$ be a subclass of $\Lin$ with an $\R^+$-valued function $\norm{\cdot}_{\U}$ satisfying:

(1) $I_{\mathcal{K}}\in \U$ with $\norm{I_{\mathcal{K}}}_{\U}=1$,

(2) $T_1, T_2, \dots\in \U(X, Y)$ and $\sum_{n=1}^\infty \norm{T_n}_\U^s<\infty$ implies $T:=\sum_{n=1}^\infty T_n\in \U (X, Y)$ with $\norm{T}^s_\U\le \sum_{n=1}^\infty \norm{T_n}^s_\U$,

(3) $S\in \Lin (X_0, X)$, $T\in \Lin (X, Y)$ and $R\in \Lin (Y, Y_0)$ imply $RTS\in \U (X_0, Y_0)$ and $\norm{RTS}_\U \le \norm{R} \norm{T}_\U \norm{S}$.

Then, $(\U, \norm{\cdot}_\U)$ is an $s$-Banach operator ideal.

\end{lemma}

\begin{proof}[Proof of Theorem \ref{theorem_guantanamo}]
 Let us see that conditions (1)-(3) of Lemma \ref{lemma_criterion} are satisfied.

 Let us obtain (1). Note that $I_\K=1\otimes 1$, so $I_\K\in \BN_{(w, p, q)}$ and $\norm{I_\K}_{(w, p, q)}\le 1$. Moreover, let $I_\K=\sum_{i=2}^\infty w_ia_i\otimes y_{i-1}$ be any $(w, p, q)$-backward nuclear representation. Then, $1=\sum_{i=2}^\infty w_ia_iy_{i-1}$ and
\[
1\le (\sum_{i=2}^\infty\abs{w_ia_iy_{i-1}})^{1/s}\le (\sum_{i=2}^\infty\abs{w_ia_iy_{i-1}}^s)^{1/s}\le \norm{w}_\infty \norm{(a_i)}^{w^*}_{\ell^{q'}}\norm{(y_i)}^w_{p'}.
\]
Hence, $1\le \norm{I_\K}_{(w, p, q)}$ and (1) is obtained.

Let us establish (2). Let $T_1, T_2, \dots\in \BN_{(w, p, q)}(X, Y)$ and $\sum_{n=1}^\infty \norm{T_n}^s_{(w, p, q)}<\infty$. Given $\varepsilon >0$, choose $(w, p, q)$-backward nuclear representations
\[
T_n=\sum_{i=2}^\infty w_{n, i}a_{n, i}\otimes y_{n, i-1}
\]
such that for every $n\ge 1$
\[
\norm{(w_{n, i})_i}_\infty \norm{(a_{n, i})_i}^{w^*}_{q'}\norm{(y_{n, i-1})_i}^w_{p'}\le (1+\varepsilon) \norm{T_n}_{(w, p, q)}.
\]
We may suppose that
\begin{align*}
&\norm{(w_{n, i})_i}_\infty\le 1\\
&\norm{(a_{n, i})_i}^{w^*}_{q'}\le \big((1+\varepsilon)\norm{T_n}_{(w, p, q)}\big)^{s/q'}\\
&\norm{(y_{n, i-1})_i}^w_{p'}\le \big((1+\varepsilon)\norm{T_n}_{(w, p, q)}\big)^{s/p'}
\end{align*}

which implies
\begin{align*}
&\norm{(w_{n, i})_{n, i}}_\infty\le 1\\
&\norm{(a_{n, i})_{n, i}}^{w^*}_{q'}\le \big((1+\varepsilon)^s \sum_{n=1}^\infty \norm{T_n}_{(w, p, q)}^s\big)^{1/q'}\\
&\norm{(y_{n, i-1})_{n, i}}^w_{p'}\le \big((1+\varepsilon)^s \sum_{n=1}^\infty \norm{T_n}_{(w, p, q)}^s\big)^{1/p'}.
\end{align*}

Consequently, $T:=\sum_{n=1}^\infty \sum_{i=2}^\infty w_{n, i}a_{n, i}\otimes y_{n, i-1}\in \BN_{(w, p, q)}(X, Y)$ and
\[
\norm{T}^s_{(w, p, q)}\le (1+\varepsilon)^s \sum_{n=1}^\infty \norm{T_n}^s_{(w, p, q)}, 
\]
and (2) is obtained.

Finally, let us obtain (3). Let $S\in \Lin (X_0, X)$, $T\in \BN_{(w, p, q)}(X, Y)$, $R\in \Lin (Y, Y_0)$ and $\varepsilon >0$. Consider the $(w, p, q)$-backward nuclear representation
\[
T=\sum_{i=2}^\infty w_ia_i\otimes y_{i-1}
\]
such that $\norm{w}_\infty \norm{(a_i)}^{w^*}_{q'}\norm{(y_i)}^w_{p'}\le (1+\varepsilon)\norm{T}_{(w, p, q)}$. Then, $RTS=\sum_{i=2}^\infty w_i S^*a_i\otimes Ry_{i-1}$ and
\[
\norm{RTS}_{(w, p, q)}\le \norm{w}_\infty \norm{(S^*a_i)}^{w^*}_{q'}\norm{(Ry_i)}^w_{p'}\le \norm{R}\norm{T}_{(w, p, q)}\norm{S}.
\]
\end{proof}

We have the following factorization property of $\BN_{(w, p, q)}$.

\begin{theorem}
Let $T\in \Lin(X, Y)$ and $B_w\in \Lin(\ell^{q'}, \ell^p)$ be a unilateral backward weighted shift. The following are equivalent:

(1) $T\in \BN_{(w, p, q)}(X, Y)$,

(2) there exists $S\in \Lin (X, \ell^{q'})$ and $R\in \Lin (\ell^p, Y)$ such that $T=RB_wS$.

Moreover, $\norm{T}_{(w, p, q)}=\inf \norm{R}\norm{B_w}\norm{S}$ where the infimum is taken over all possible factorizations.
\end{theorem}

\begin{proof}
Let $T\in \BN_{(w, p, q)}(X, Y)$ and $\varepsilon >0$. Consider a backward nuclear representation $T=\sum_{i=2}^\infty w_i a_i\otimes y_{i-1}$ such that $\norm{w}_\infty \norm{(a_i)}^{w*}_{q'}\norm{(y_i)}^w_{p'}\leq (1+\varepsilon) N_{(w, p, q)}(T)$. Consider the operators $S\in \Lin(X, \ell_{q'})$ and $R\in \Lin (\ell_p, Y)$ defined as $Sx=(a_i(x))_i$ for any $x\in X$ and $R((\nu_i))=\sum_{i=1}^\infty \nu_i y_i$ for any $(\nu_i)_i\in \ell_p$. It is well-known that $\norm{S}=\norm{(a_i)}^{w*}_{q'}$ and $\norm{T}=\norm{(y_i)}^w_{p'}$. Hence, $\norm{B_w}_\infty \norm{S}\norm{T}\leq (1+\varepsilon) N_{(w, p, q)}(T)$ and $RB_wS=T$. Conversely, by definition $B_w=\sum_{i=2}^\infty w_ie_i^* \otimes e_{i-1}\in \BN_{(w, p, q)}(\ell^{q'}, \ell^p) $ with $N_{(w, p, q)}(B_w)\leq \norm{w}_\infty=\norm{B_w}$, where $e_i^*(e_i)=1$ and $e_i^*(e_j)=0$ for all $i\neq j$. Since $\BN_{(w, p, q)}$ is an operator ideal, we have that $T=RB_wS=\sum_{i=2}^\infty w_i S^*e_i^*\otimes R e_{i-1}\in \BN_{(w, p, q)}(X, Y)$ and 
\[
N_{(w, p, q)}(T)\leq \norm{R}N_{(w, p, q)}(B_w)\norm{S}\leq \norm{R}\norm{B_w}\norm{S}.
\]

\end{proof}

\subsection {The injective hull of the operator ideal $(\BN_{(w, p, q)}, \norm{\cdot}_{(w, p, q)})$}
 Let us focus on the injective hull of  the operator ideal  $(\BN_{(w, p, q)}, \norm{\cdot}_{(w, p, q)})$.

 Let $X$ be a Banach space, denote $X^\infty:=\ell_\infty(U_{X^*})$ and the injection $J_X: X\to X^\infty$ defined as $J_X x:=(a(x))_{a\in U_{ X^*}}$.

\begin{theorem}
\label{theorem_istanbul}
Let $T\in \Lin(X, Y)$ and $B_w\in \Lin(\ell^{q'}, \ell^p)$ be a unilateral backward weighted shift. The following are equivalent:

(1) $T\in \BN_{(w, p, q)}^{inj}(X, Y)$,

(2) there exist a closed subspace $F$ of $\ell^p$, $S\in \Lin(X, \ell^{q'})$ and $R\in \Lin(F, Y)$ such that $T=R B_w S$. Moreover, 
\begin{equation}
\label{eq_giotto}
\norm{T}^{inj}_{(w, p, q)}=\inf \norm{R}\norm{B_w}\norm{S},
\end{equation}
where the infimum is taken over all possible factorizations, as described above.
\end{theorem}

\begin{proof}
 Let $T\in \BN_{(w, p, q)}^{inj}(X, Y)$, then $J_YT\in \BN_{(w, p, q)}(X, Y)$. Let $\varepsilon >0$, consider $S\in \Lin (X, \ell^{q'})$ and  $R\in \Lin (\ell^p, Y^\infty)$ such that $J_YT=RB_wS$ with 
\[
\norm{R}\norm{B_w}\norm{S}\le \norm{J_YT}_{(w, p, q)}+\varepsilon = \norm{T}^{inj}_{(w, p, q)}+\varepsilon.
\]

Define $H|_{B_wS(X)}:=J_Y^{-1}R|_{B_wS(X)}$. Let $F$ be the closure of $B_wS(X)$, and denote by $\bar{H}$ the continuous extension of $H$ to $F$. Note that $T=\bar{H}B_wS$ and
\begin{equation}
\label{eq_morandi}
\norm{\bar{H}}\norm{B_wS}\le \norm{H}\norm{B_w}\norm{S}=\norm{R}\norm{B_w}\norm{S}\le \norm{T}^{inj}_{(w, p, q)}+\varepsilon.
\end{equation}

 Conversely, let $T\in \Lin(X, Y)$ such that $T=RB_wS$ where $S\in \Lin(X, \ell^{q'})$, $R\in \Lin(F, Y)$ and $F$ is a closed subspace of $\ell^p$. By the extension property of $Y^\infty$ there exists $R^\infty\in \Lin(\ell^p, Y^\infty)$ such that $R^\infty x=J_YRx$, for any $x\in F$ and $\norm{R^\infty}=\norm{J_YR}$. Thus, $J_YT=R^\infty B_wS$. Then, $J_YT\in \BN_{(w, p, q)}(X, Y^\infty)$ and $T\in \BN_{(w, p, q)}^{inj}(X, Y)$. Moreover,
\begin{equation}
\label{eq_mantegna}
\norm{T}^{inj}_{(w, p, q)}=\norm{J_YT}_{(w, p, q)}\le \norm{R^\infty}\norm{B_w}\norm{S}=\norm{R}\norm{B_w}\norm{S}.
\end{equation}
Putting together (\ref{eq_morandi}) and (\ref{eq_mantegna}) we obtain (\ref{eq_giotto}).
\end{proof}

\begin{theorem}
\label{theorem_havana}
Let $T\in \Lin(X, Y)$ and $B_w\in \Lin(\ell^{q'}, \ell^p)$ be a unilateral backward weighted shift. The following are equivalent:

(1) $T\in \BN_{(w, p, q)}^{inj}(X, Y)$,

(2) there exists $(x_n)\in \ell^{w^*}_{q'}(X^*)$ such that 
\[
\norm{Tx}\le \norm{B_w(x_n(x))_n}_p, \quad \forall x\in X.
\]
In addition, $\norm{T}^{inj}_{(w, p, q)}=\inf \norm{w}_\infty \norm{(x_n)_n}^{w*}_{q'}$, where the infimum is taken over all sequences $(x_n)_n$ satisfying condition (2).
\end{theorem}

\begin{proof}
Let $T\in \BN_{(w, p, q)}^{inj}(X, Y)$ and $\varepsilon >0$. Consider a $(w, p, q)$-backward nuclear representation of $J_YT$ as $J_YT=\sum_{i=2}^\infty w_ia_i\otimes y_{i-1}$ such that 
\[
\norm{w}_\infty\norm{(a_i)}^{w^*}_{q'}\norm{(y_i)}^w_{p'}\le \norm{J_YT}_{(w, p, q)}+\varepsilon.
\]
Then, 
\begin{align*}
\norm{Tx}&\le \norm{J_YTx}= \sup_{\norm{y^*}_{(Y^\infty)^*}\le 1}\abs{y^*(J_YTx)}=\sup_{\norm{y^*}_{(Y^\infty)^*}\le 1}\abs{\sum_{i=2}^\infty w_ia_i(x)\cdot y^*(y_{i-1})}\\
&\le \norm{(y_{i-1})}^w_{p'}\cdot (\sum_{i=2}^\infty \abs{w_ia_i(x)^p})^{1/p}=\norm{B_wx_n(x)}_p,
\end{align*}
where $x_n=\norm{(y_{i-1})}^w_{p'}\cdot a_n\in \ell^{w^*}_{q'}(X^*)$. Moreover, $\norm{w}_\infty\norm{(x_n)}^{w^*}_{q'}\le \norm{T}^{inj}_{(w, p, q)}+\varepsilon$.

Conversely, suppose that condition (2) holds. Define $S\in \Lin (X, \ell^{q'})$ such that $Sx=(x_n(x))_n$, for all $x\in X$. Then, $\norm{S}=\norm{(x_n)}^{w^*}_{q'}$. By hypothesis, $\norm{Tx}\le \norm{B_wSx}_p$, for all $x\in X$.

Thus, $R:B_wS(X)\to Y$ defined as $R(B_wSx)=Tx$ is a bounded and linear operator with $\norm{R}\le 1$. Let $\bar{R}$ be the continuous linear extension of $R$ to $F$, where $F$ denotes the closure of $B_wS(X)$. Then, $\norm{\bar{R}}\le 1$. Note that $T=\bar{R}B_wS$ where $\bar{R}\in \Lin(F, Y)$. By Theorem \ref{theorem_istanbul} we have $T\in \BN^{inj}_{(w, p, q)}$ and 
\[
\norm{T}^{inj}_{(w, p, q)}\le \norm{B_wS}\norm{\bar{R}}\le \norm{w}_\infty\norm{S}=\norm{w}_\infty \norm{(x_n)}^{w^*}_{q'}.
\]
Hence,
\[
\norm{T}^{inj}_{(w, p, q)}=\inf\{\norm{w}_\infty \norm{(x_n)}^{w^*}_{q'}: \norm{Tx}\le \norm{B_w (x_n(x))_n}_p, x\in X.\}
\]
\end{proof}

\subsection {The surjective hull of the operator ideal $(\BN_{(w, p, q)}, \norm{\cdot}_{(w, p, q)})$}
 Let us turn our attention to the surjective hull of  the operator ideal  $(\BN_{(w, p, q)}, \norm{\cdot}_{(w, p, q)})$.

 Let $X$ be a Banach space, denote $X^1:=\ell_1(U_{X})$ and the surjection $Q^1_X: X^1\to X$ defined as $Q^1_X (\epsilon_x):=\sum_{x\in U_X}\epsilon_x x$, for $(\epsilon_x)\in \ell_1(U_X)$.

\begin{theorem}
\label{theorem_izmir}
Let $T\in \Lin(X, Y)$ and $B_w\in \Lin(\ell^{q'}, \ell^p)$ be a unilateral backward weighted shift. The following are equivalent:

(1) $T\in \BN_{(w, p, q)}^{sur}(X, Y)$,

(2) there exist a quotient space $Z$ of $\ell^{q'}$, $S\in \Lin(X, Z)$ and $Q\in \Lin(Z, Y)$ such that $T=QS$. Moreover, $RB_w=Q\pi$ with $\pi:\ell^{q'}\to Z$ the quotient operator and $R\in \Lin (\ell^p, Y)$.

In addition, $\norm{T}^{sur}_{(w, p, q)}=\inf \norm{R}\norm{B_w}\norm{S}$, where the infimum is taken over all possible factorizations, as described above.
\end{theorem}

\begin{proof}
 Let $T\in \BN_{(w, p, q)}^{sur}(X, Y)$, then $TQ^1_X\in \BN_{(w, p, q)}(X^1, Y)$. Let $\varepsilon >0$ and consider $P\in \Lin (X^1, \ell^{q'})$, $R\in \Lin (\ell^p, Y)$ with $TQ^1_X=RB_wP$ and 
\[
\norm{R}\norm{B_w}\norm{P}\le \norm{TQ^1_X}+\varepsilon=\norm{T}^{sur}_{(w, p, q)}+\varepsilon.
\]

Set $Z:=\ell^{q'}/(RB_w)^{-1}(0)$ and define the one-to-one operator $Q:Z\to Y$ as $Q((\epsilon_n)+(RB_w)^{-1}(0))=RB_w(\epsilon_n)_n$, for any $(\epsilon_n)\in \ell^{q'}$. Note that, 
\[
\norm{Q((\epsilon_n)+(RB_w)^{-1}(0))}=\norm{RB_w(\epsilon_n)_n}\le \norm{RB_w}\norm{(\epsilon_n)_n+R^{-1}(0)}, \quad \forall (\epsilon_n)_n\in \ell^p.
\]
Hence, $\norm{Q}\le \norm{RB_w}$. Define, $S\in \Lin (X, Z)$ as $S=Q^{-1}T$. Note that $Q^{-1}TQ^1_X=Q^{-1}RB_wP$. Then, 
\[
\norm{S}=\norm{SQ^1_X}\le \norm{Q^{-1}RB_w}\norm{P}\le \norm{P}.
\]
Hence, $T=QS$ and 
\begin{equation}
\label{eq_davinci}
\norm{S}\norm{R}\norm{B_w}\le \norm{P}\norm{R}\norm{B_w}\le \norm{T}^{sur}_{(w, p, q)}+\varepsilon.
\end{equation}

Conversely, let $T+QS$ such that $S\in \Lin (X, Z)$, $Q\in \Lin (Z, Y)$ with $Z$ a quotient space of $\ell^{q'}$. Moreover, let $R\in \Lin (\ell^p, Y)$ such that $RB_w=Q\pi$, where $\pi:\ell^{q'}\to Z$ is the quotient operator. By the lifting property of $X^1$, for each $\varepsilon >0$ there exists $S_0\in \Lin (X^1, \ell^{q'})$ such that $\pi S_0=SQ^1_X$, and 
\[
\norm{S_0}\le \Big(1+\frac{\varepsilon}{\norm{S}\norm{R}\norm{B_w}}\Big)\norm{\pi}\norm{SQ^1_X}\le \norm{S}+\frac{\varepsilon}{\norm{R}\norm{B_w}}.
\]
Note that $TQ^1_X=QSQ^1_X=Q\pi S_0=RB_wS_0$. Then, $TQ^1_X\in \BN_{(w, p, q)}(X^1, Y)$, i.e. $T\in \BN^{sur}_{(w, p, q)}(X, Y)$. Moreover, 
\begin{align*}
\norm{T}^{sur}_{(w, p, q)}&=\norm{TQ^1_X}_{(w, p, q)}\le \norm{R}\norm{B_w}\norm{S_0}\\
&\le \norm{R}\norm{B_w}\Big(\norm{S}+\frac{\varepsilon}{\norm{R}\norm{B_w}}\Big)\\
&\le \norm{R}\norm{B_w}\norm{S}+\varepsilon.
\end{align*}
This, together with (\ref{eq_davinci}) lead to $\norm{T}^{sur}_{(w, p, q)}=\inf \norm{R}\norm{B_w}\norm{S}$, where the infimum is taken over all possible factorizations as in (2).
\end{proof}

We will need the following.

\begin{lemma}[subsection 8.5.4 \cite{PieID}]
\label{lemma_azul}
Let $(\mathcal{U}, \norm{\cdot}_{\mathcal{U}})$ be a normed operator ideal. Let $S_0\in \Lin(E_0, F)$ and $S\in \mathcal{U}(E, F)$. If $S_0(U_{E_0})\subseteq S(U_E)$, then $S_0\in \mathcal{U}^{sur}(E_0, F)$ and $\norm{S_0}^{sur}_{\mathcal{U}}\le \norm{S}_{\mathcal{U}}$.
\end{lemma}

\begin{theorem}
\label{theorem_miami}
Let $T\in \Lin(X, Y)$ and $B_w\in \Lin(\ell^{q'}, \ell^p)$ be a unilateral backward weighted shift. The following are equivalent:

(1) $T\in \BN_{(w, p, q)}^{sur}(X, Y)$,

(2) there exists $P\in \BN_{(w, p, q)}(\ell^{q'}, Y)$ of the form $Px:=\sum_{i=1}^\infty x_{i+1}w_{i+1}y_i$ with $(y_i)\in \ell^w_q(Y)$ satisfying 
$$
T(U_X)\subseteq \left\{Px: (x_n)\in U_{\ell^{q'} }\right\}.
$$
Moreover, 
\[
\norm{T}^{sur}_{(w, p, q)}=\inf \big\{\norm{P}_{(w, p, q)}: T(U_X)\subseteq P(U_{\ell^{q'}}), P\in \BN_{(w, p, q)}(\ell^{q'}, Y)\big\}.
\]
\end{theorem}

\begin{proof}
Let $T\in \BN_{(w, p, q)}^{sur}(X, Y)$ and $\varepsilon >0$. By Theorem \ref{theorem_izmir}, there exists a quotient space $Z$ of $\ell^{q'}$, $S\in \Lin (X, Z)$, $Q\in \Lin (Z, Y)$ and $R\in \Lin (\ell^p, Y)$ such that 
\[
T=QS, \quad Q\pi=RB_w, \quad \norm{R}\norm{B_w}\norm{S}\le \norm{T}^{sur}_{(w, p, q)}+\varepsilon.
\]
 Let $\delta >0$, by the lifting property of $X^1$, there exists $S_0\in \Lin (X^1, \ell^{q'})$ such that $\pi S_0=SQ^1_X$, and
\[
\norm{S_0}\le (1+\delta)\norm{\pi}\norm{SQ^1_X}\le (1+\delta)\norm{S}.
\]
 Set $P:=\norm{S_0}Q\pi \in \Lin (\ell^{q'}, Y)$. Observe that, $P=RB_w(\norm{S_0}\cdot I_{\ell^{q'}})$, where $I_{\ell^{q'}}$ denotes the identity on $\ell^{q'}$. Hence, $P\in \BN_{(w, p, q)}(\ell^{q'}, Y)$ with $\norm{P}_{(w, p, q)}\le \norm{P}\norm{B_w}\norm{S_0}$.
 Since $Q^1_X$ is surjective, for any $x\in U_X$ there exists $x_1\in X^1$ such that $Q^1_Xx_1=x$. Then, 
\[
Tx=TQ^1_Xx_1=QSQ^1_X x_1=Q\pi S_0x_1=P\Big(\frac{1}{\norm{S_0}}\cdot S_0x_1\Big).
\]
In other words, $T(U_X)\subseteq P(U_{\ell^{q'}})$. Finally, 
\[
\norm{P}_{(w, p, q)}\le (1+\delta)\norm{S}\norm{R}\norm{B_w}\le (1+\delta)(\norm{T}^{sur}_{(w, p, q)}+\varepsilon).
\]
 Conversely, suppose the condition (2) holds, then we conclude by Lemma \ref{lemma_azul} that $T\in \BN^{sur}_{(w, p, q)}(X, Y)$ with $\norm{T}^{sur}_{(w, p, q)}\le \norm{P}_{(w, p, q)}$.
\end{proof}

It is worth noticing here the regularity of $\BN^{sur}_{(w, p, q)}$. This result can be proven by following the same argument given by Pietsch in Lemma 3 \cite{Pie}, where he showed the regularity of the surjective hull of the ideal of $(r, p, q)$-nuclear operators. So, we have the following:

\begin{proposition}
\label{morning}
 The ideal $\BN^{sur}_{(w, p, q)}$ is regular.
\end{proposition}

The following is proved in \cite{Pie}, we will need it later.

\begin{lemma}[Lemma 2 \cite{Pie}]
\label{afternoon}
 Let $\U$ be any operator ideal, then $(\U^{reg})^{sur}=(\U^{sur})^{reg}$.
\end{lemma}

\subsection{Connection between the injective and surjective hull of $\BN_{(w, p, q)}$}
 It is well-known that the dual operator ideal of every surjective operator ideal is injective, see Proposition 1, section 4.7 \cite{PieID}.  More precisely, for the operator ideal $\BN_{(w, p, q)}$ we have the following.

\begin{theorem}
\label{theorem_orlando}
$(\BN^{sur}_{(w, p, p')})^{dual}=\BN^{inj}_{(w, p', p)}$.
\end{theorem}

\begin{proof}
Note that,
\[
\BN^{inj}_{(w, p', p)}=\Big(\BN^{reg}_{(w, p', p)}\Big)^{inj}=\Big(\BN^{dual}_{(w, p, p')}\Big)^{inj}\subseteq \Big(\BN^{sur}_{(w, p, p')}\Big)^{dual}.
\]
 Conversely, if $T\in  \Big(\BN^{sur}_{(w, p, p')}\Big)^{dual}(X, Y)$, then $T^*\in \BN^{sur}_{(w, p, p')}(Y^*, X^*)$. Thus, 
\[
T^*(U_{Y^*})\subseteq\Big\{ \sum_{n=1}^\infty \lambda_{n+1}w_{n+1}y_n: (\lambda_n)\in U_{\ell_p}\Big\}, \mbox{ with } (y_n)\in \ell^w_{p'}(X^*).
\]
By Hahn-Banach Theorem, there exists $y^*\in U_{Y^*}$ such that
\begin{align*}
\norm{Tx}&=\abs{y^*(Tx)}=\abs{(T^*y^*)x}=\abs{(\sum_{n=1}^\infty \lambda_{n+1}w_{n+1}y_n)x}\\
&\le (\sum_{n=1}^\infty \abs{\lambda_{n+1}}^p)^{1/p} (\sum_{n=1}^\infty\abs{(w_{n+1}y_n)x}^{p'})^{1/p'}\\
&\le \norm{B_w(x_n(x))_n}_{p'}, \mbox{ for some } (x_n)\in \ell^{w^*}_{p'}(X^*).
\end{align*}
 Hence, $T\in \BN^{inj}_{(w, p', p)}(X, Y)$ by Theorem \ref{theorem_havana}.
\end{proof}

 It is also known that the dual operator ideal of every injective operator ideal is surjective, see Proposition 2, section 4.7 \cite{PieID}. More precisely, for the operator ideal $\BN_{(w, p, q)}$ we have the following.

\begin{theorem}
\label{theorem_westpalmbeach}
$(\BN^{inj}_{(w, p, p')})^{dual}=\BN^{sur}_{(w, p', p)}$.
\end{theorem}

\begin{proof}
By Proposition \ref{morning}, Lemma \ref{afternoon}, 18.1.6 \cite{PieID} and 4.7.16 \cite{PieID}, we have

\[\BN^{sur}_{(w, p', p)}=(\BN^{sur}_{(w, p', p)})^{reg}=(\BN^{reg}_{(w, p', p)})^{sur}=(\BN^{dual}_{(w, p, p')})^{sur}=(\BN^{inj}_{(w, p, p')})^{dual}.
\]
\end{proof}

 A linear and bounded operator $T$ acting on $X$ is said to be \emph{dual hypercyclic} when both $T$ and its adjoint $T^*:X^*\to X^*$ are hypercyclic. A necessary condition for the existence of dual hypercyclic  operator on $X$ is that both $X$ and $X^*$ are separable, thus $\ell^1$ does not support such an operator. Also it is  known that any infinite dimensional Banach space $X$ with shrinking symmetric basis  admits a dual hypercyclic operator  \cite{Pet}. In this way, our work has connections with dual hypercyclicity of operators belonging either to the injective or the surjective hull of $\BN_{(w, p, q)}$. 

 \bibliographystyle{amsplain}

\begin{thebibliography}{99}

\bibitem{Asu} A. G. Aksoy, \textit{On a theorem of Terzio\u{gl}u}, Turk J Math, 43, (2019), 258-267.

\bibitem{BaMa} F. Bayart and E. Matheron, \textit{Dynamics of linear operators}, Cambridge Tracts in Mathematics, 179. Cambridge University Press, Cambridge (2009).

  \bibitem{Gr} A. Grothendieck, \textit{Produits tensoriels topologiques et espaces nuclearies}, Mem. Amer. Math. Soc., \textbf{1955}, (1955) No. 16, 140pp (French).
  \bibitem{Gr1} A. Grothendieck, \textit{Sur certaines classes de suites dans les espaces de Banach  et le theoreme de Dvoretzky-Rogers} , Bol., Soc., Math Sao Paulo 8, (1956), 81-110.

 \bibitem{Jon} W. B.  Johnson, \textit{Factoring compact operators}. Israel J. Math. 9 (1971), 337-345. 

\bibitem{Pet} H. Peterson, \textit{Spaces that admit hypercyclic operators with hypercyclic adjoints}, Proc. AMS, \textbf{134}, number 6,  (2005), 1671-1676.
 \bibitem{PieID} A.~Pietsch, \textit{Operator ideals}, North-Holland, Amsterdam, 1980. 
\bibitem{Pie} A. Pietsch, \textit{The ideal of $p$-compact operators and its maximal hull}, Proc. Amer. Math. Soc., \textbf{142} (2014), 519-530.
\bibitem{Sal} H. Salas,\textit{ Hypercyclic weighted shifs}. Trans. Amer. Soc. \textbf{347}, (1995),  $994-1004$.

  \bibitem{Ter1} T. Terzio\u{g}lu,  \textit{A characterization of compact linear mappings},  Arch. Math. \textbf{22} (1971), 76-78.
\end{thebibliography}

\end{document}